\documentclass[12pt]{article}
\usepackage{amsmath,amssymb,amstext,fontenc,graphicx,theorem,hyperref}
\usepackage[utf8]{inputenc}
\usepackage{tikz}

\usepackage[letterpaper]{geometry}
\usepackage{relsize}
\usepackage{sectsty}
\allsectionsfont{\larger[-1]} 

\renewcommand\paragraph{\@startsection{paragraph}{4}{\z@}
                                    {2ex \@plus.5ex \@minus.2ex}
                                    {-1em}
                                    {\normalfont\normalsize\bfseries}}

\renewcommand\subparagraph{\@startsection{subparagraph}{5}{\parindent}
                                       {2ex \@plus.5ex \@minus .2ex}
                                       {-1em}
                                      {\normalfont\normalsize\bfseries}}

\newlength{\BiblioSpacing}
\renewenvironment{thebibliography}[1]{
\begin{oldthebibliography}{#1}
\setlength{\parskip}{\BiblioSpacing}
\setlength{\itemsep}{\BiblioSpacing}
}
{
\end{oldthebibliography}
}

\usepackage[strict]{changepage}
\def\abstractname{Abstract}
\def\abstract{\begin{adjustwidth}{1cm}{1cm} \par \footnotesize \noindent {\bf \abstractname:} 
\def\endabstract{ \end{adjustwidth} \smallskip }}

{\theorembodyfont{\itshape}\newtheorem{theorem}{Theorem}[section]}
{\theorembodyfont{\itshape}}
{\theorembodyfont{\itshape}\newtheorem{defi}[theorem]{Definition}}
{\theorembodyfont{\itshape}\newtheorem{lemma}[theorem]{Lemma}}
{\theorembodyfont{\itshape}}
{\theorembodyfont{\rm}}
{\theorembodyfont{\rm}\newtheorem{rem}[theorem]{Remark}}
{\theorembodyfont{\rm}}
{\theorembodyfont{\rm }}
\numberwithin{equation}{section}

\title{\Large\bf Upper and Lower Solution Method for Regular Discrete Second-Order Single-Variable BVPs}
\author{\sc S. Bandyopadhyay, K. Byassee, and C. Lynch}
\date{}

\begin{document}
\maketitle

\begin{abstract}
This paper investigates the existence of positive solutions for regular discrete second-order single-variable boundary value problems with mixed boundary conditions, including a nonhomogeneous Dirichlet boundary condition, of the form:
\begin{equation*}
u^{\Delta \Delta}(t-1)+h(t,\ u(t),\ u^{\Delta}(t-1))=0 \mbox{ for }t\in[1,\ T+1];~u^{\Delta}(0)=0;~u(T+2)=g(T+2)    
\end{equation*}
where h is continuous on $[1, T + 1] \times \mathbb{R}^2$ and $g: [0, T + 2] \to  \mathbb{R}^+$ is continuous. Using the concept of upper and lower solutions, we establish conditions under which the boundary value problem admits at least one positive solution. Our approach involves constructing an auxiliary problem with a modified nonlinearity and applying Brouwer Fixed Point Theorem to a carefully defined solution operator. We prove that any solution to this auxiliary problem that remains within the bounds of the upper and lower solutions is equivalent to a solution of the original problem. 
\end{abstract}


\begin{MSC}
39A27
\end{MSC}

\section{Introduction}

Discrete boundary value problems (BVPs) represent a fundamental area of research at the intersection of difference equations and mathematical modeling, with wide-ranging applications across disciplines. Although continuous BVPs have been extensively studied for centuries, their discrete counterparts have gained significant attention in recent decades due to their natural alignment with numerical methods and digital computation. This paper investigates positive solutions for a class of regular discrete second-order boundary value problems with mixed boundary conditions, specifically addressing non-homogeneous Dirichlet boundary conditions.

The study of discrete second-order boundary value problems is motivated by both theoretical interest and practical applications. From a theoretical perspective, these problems present unique challenges that do not manifest in continuous systems, including phenomena related to the discrete topology of the underlying domain. Practically, such boundary value problems arise naturally in the mathematical modeling of various physical and engineering systems, including mechanical vibrations, heat transfer across discrete media, population dynamics with seasonal variations, and quantum mechanics on lattice structures.

The specific boundary value problem we consider takes the form:

$$u^{\Delta\Delta}(t-1)+h(t, u(t), u^\Delta(t-1)) = 0 \text{ for } t \in [1, T+1]; u^\Delta(0) = 0; u(T+2) = g(T+2)$$

where $h$ is continuous on $[1, T+1] \times \mathbb{R}^2$ and $g: [0, T+2] \to \mathbb{R}^+$ is continuous. Here, $u^\Delta$ is the first order forward difference of $u$, $u^{\Delta\Delta}$ is the second order forward difference of $u$ (for specific definition see \ref{def:second-deriv}) and the Neumann-type condition $u^\Delta(0) = 0$ represents a zero-slope requirement at the left boundary, while the Dirichlet condition $u(T+2) = g(T+2)$ prescribes a specific value at the right boundary. This mixed boundary condition structure appears in various physical contexts, such as heat flow problems with insulation at one end and a fixed temperature at the other.

Recent work by Kunkel et al.~\cite{Ku1, KuThesis, KuLa} explored similar boundary value problems but focused primarily on homogeneous boundary conditions. Our approach extends their methodology by incorporating non-homogeneous boundary conditions while employing the powerful machinery of upper and lower solutions. This technique has proven effective in both continuous and discrete settings, offering a constructive approach to demonstrating existence results.

The well known method of upper and lower solutions has been adapted to the discrete setting by several researchers, including Agarwal~\cite{AgOR1} and Zeidler~\cite{Ze}. Our contribution builds upon this foundation by establishing conditions under which the given boundary value problem admits at least one positive solution. The positivity requirement is particularly relevant in applications where the unknown function represents physical quantities that cannot meaningfully take negative values, such as population densities, concentrations, or energy levels.

The structure of this paper is as follows: Section 2 provides essential preliminaries from time scale calculus, establishing the mathematical framework for our analysis. Section 3 formally introduces the problem, defines the concepts of upper and lower solutions, and states our main existence theorem. Section 4 contains the proof of the main result, employing a carefully constructed auxiliary problem and applying Brouwer's Fixed Point Theorem to a solution operator. We conclude with a discussion of potential extensions and applications in Section 5.

Our results not only contribute to the theoretical understanding of discrete boundary value problems but also provide practical insights for numerical approximations of continuous problems and for directly modeling phenomena that are inherently discrete in nature. The constructive approach developed here may further serve as a foundation for computational algorithms aimed at approximating solutions to these problems.

\section{Preliminaries}

Prior to stating our problem, we will introduce some basic definitions from time scale calculus. We will use these definitions throughout the remainder of the paper. See~\cite{BoPe} for further details.

\begin{defi}
Let \(T\in\mathbb{N}\) be fixed. Define the discrete intervals
$$\mathbb{T} := [0,\ T+2] = \{0,\ 2,\ \dots,\ T,\ T+2\},$$
and
$$\mathbb{T}^{\circ} := [1,\ T+1] = \{1,\ 2,\ \dots,\ T,\ T+1\}.$$
\end{defi}

\begin{rem}
    Observe that, $\mathbb{T}^{\circ}= \mathbb{T} \setminus \{0,T+2\} $
\end{rem}

\begin{defi}
\label{def:second-deriv}
For the function $u : \mathbb{T} \rightarrow \mathbb{R},$ define the delta derivative, $u^{\Delta},$ by
$$u^{\Delta} ({i}) := \frac{u({i+1})-u({i})}{{i+1}-{i}}=u({i+1})-u({i}), \quad i \in \mathbb{T}^{\circ} \cup \{{0}\}.$$
We make note that $u^{\Delta \Delta}(t_{i}) = \left( u^{\Delta} \right)^{\Delta} (t_{i}).$
\end{defi}

\section{Problem}

We consider the following difference equation:
\begin{equation}
    \label{Problem}
    \begin{cases}
        u^{\Delta \Delta}(t-1)+h(t,\ u(t),\ u^{\Delta}(t-1))=0\text{,\quad}t\in[1,\ T+1]\\
        u^{\Delta}(0)=0\\
        u(T+2)=g(T+2)
    \end{cases}
\end{equation}
where
\begin{itemize}
    \item[{\bf(A1)}]\(h\) is continuous on \([1,\ T+1]\times\mathbb{R}^2\) i.e. $h(\cdot, x, y)$ is defined on $[1,T+1]$ for each $(x,y) \in \mathbb{R}^2$ and $h(t, \cdot, \cdot)$ is continuous on $\mathbb{R}^2$ for each $t \in [1, T+1]$.
    \item[{\bf(A2)}]\(h\) is non-increasing in its third variable.
    \item[{\bf(A3)}] \(g:[0,\ T+2] \to \mathbb{R}^+\) is continuous.
\end{itemize}

\begin{defi}
A function \(u\colon[0,\ T+2]\to\mathbb{R}\)  is said to be a solution of \eqref{Problem} if and only if $u$ satisfies the difference equation \eqref{Problem} along with the boundary condition.
\end{defi}

\begin{defi}
A solution $u$ is said to be a positive solution of \eqref{Problem} if \(u(t)>0\) for \(t\in[1,\ T+1]\). 
\end{defi}

\begin{defi}
 \(\alpha\colon[0,\ T+2]\to\mathbb{R}\) is called a lower solution of \eqref{Problem} if
\begin{equation}
    \label{alpha}
    \begin{cases}
        \alpha^{\Delta\Delta}(t-1)+h(t,\ \alpha(t),\ \alpha^{\Delta}(t-1))\geq0\text{, for }t\in[1,\ T+1]\\
        \alpha^{\Delta}(0)\geq0\\
        \alpha(T+2)\leq g(T+2)
    \end{cases}
\end{equation}   
\end{defi} 

\begin{defi}
 \(\beta\colon[0,\ T+2]\to\mathbb{R}\) is called 
an upper solution of \eqref{Problem} if
\begin{equation}
    \label{beta}
    \begin{cases}
        \beta^{\Delta\Delta}(t-1) + h(t,\ \beta(t),\ \beta^{\Delta}(t-1))\leq0\text{, for }t\in[1,\ T+1]\\
        \beta^{\Delta}(0)\leq0\\
        \beta(T+2)\geq g(T+2)
    \end{cases}
\end{equation}   
\end{defi} 

Our Goal is to prove existence of positive solutions for \eqref{Problem}, in particular, we will prove the following Theorem:

\begin{theorem}
\label{Theorem}
 Let \(\alpha\) and \(\beta\) be a lower and an upper
function of \eqref{Problem} and \(\alpha\leq\beta\) on \([1,\ T+1]\) and $h, g$ satisfy the assumptions {\bf(A1)}-{\bf(A3)}. Then \eqref{Problem} has a solution \(u(t)\) satisfying:
\begin{equation}
    \label{upper_and_lower}
    \alpha(t)\leq u(t)\leq\beta(t)\text{,\quad}t\in[0,\ T+2]
\end{equation}   
\end{theorem}

\section{Proof of Theorem \ref{Theorem}}

\subsection*{Step 1: Construction of Auxiliary Problem} 

In this step, we introduce the auxiliary problem corresponding to \eqref{Problem}, by defining the auxiliary function $\tilde{h}$ corresponding to the nonlinearity $h$. First, for \(t\in[1,\ T+1],\ x,\ z\in\mathbb{R}\), we define,
\begin{align}
    \label{sigma}
    \sigma(t,\ z) &=
    \begin{cases}
        \beta(t-1)\text{,} & z>\beta(t-1)\\
        z\text{,} & \alpha(t-1)\leq z\leq\beta(t-1)\\
        \alpha(t-1)\text{,} & z<\alpha(t-1)
    \end{cases}\\[2ex] 
    \end{align}
   Next, we define, 
    
    \begin{align}
    \label{h_tilde}    
    \tilde{h}(t,\ x,\ x-z) &=
    \begin{cases}
        h(t,\ \beta(t),\ \beta(t)-\sigma(t,\ z))-\dfrac{x-\beta(t)}{x-\beta(t)+1}
        \text{,} & x>\beta(t)\\
        h(t,\ x,\ x-\sigma(t,\ z))\text{,} & \alpha(t)\leq x\leq\beta(t)\\
        h(t,\ \alpha(t),\ \alpha(t)-\sigma(t,\ z))+\dfrac{\alpha(t)-x}{\alpha(t)-x+1}
        \text{,} & x<\alpha(t)
    \end{cases}
\end{align}
where \(\sigma(t,\ z)\) is defined as above in \eqref{sigma}.
\begin{figure}[h]
\setlength{\unitlength}{.6cm}
\begin{center}
\begin{tikzpicture}
    \draw[->] (-3,0) -- (3,0); 
    \draw[->] (0,-2) -- (0,2); 

    \draw[thick, domain=-3:3, smooth] plot (\x, {1.5*exp(-0.3*\x*\x)*sin(180*\x/2)});
    
    \draw[dashed] (-3,1) -- (3,1) node[right] {\small $\beta(t)$};
    \draw[dashed] (-3,-1) -- (3,-1) node[right] {\small $\alpha(t)$};

    \node[right] at (2,.5) {\small $h(t)$};

\end{tikzpicture}
\end{center}
\end{figure}
\begin{figure}[h]
\setlength{\unitlength}{.6cm}
\begin{center}
\begin{tikzpicture}
    \draw[->] (-3,0) -- (3,0); 
    \draw[->] (0,-2) -- (0,2); 

    \draw[dashed] (-3,1) -- (3,1) node[right] {\small $\beta(t)$};
    \draw[dashed] (-3,-1) -- (3,-1) node[right] {\small $\alpha(t)$};

    \draw[thick, domain=-3:3, smooth, samples=100] 
        plot (\x, {max(-1, min(1.5*exp(-0.3*\x*\x)*sin(180*\x/2), 1))});
    
    \node[right] at (2,0.5) {\small $\Tilde{h}(t)$};

\end{tikzpicture}
\end{center}
\end{figure}

Observe that, \(\tilde{h}\) is continuous at \(x=\alpha(t)\) and \(x=\beta(t)\). Indeed,
\begin{align*}
    &\lim_{x\to\beta(t)}\tilde{h}(t,\ \beta(t),\ \beta(t)-\sigma(t,\ z))
    -\frac{x-\beta(t)}{x-\beta(t)+1}=h(t,\ x,\ x-\sigma(t,\ z))\\
    &\text{and}\\
    &\lim_{x\to\alpha(t)}\tilde{h}(t,\ \alpha(t),\ \alpha(t)-\sigma(t,\ z))
    +\frac{\alpha(t)-x}{\alpha(t)-x+1}=h(t,\ x,\ x-\sigma(t,\ z))\text{.}
\end{align*}
Therefore, we can conclude that, $\tilde{h}(t)$ is continuous in the discrete time scale \(\mathbb{T}^0=[1,\ T+1]\). Now, since $\mathbb{T}^0$ is closed and bounded, hence, we can apply extreme value theorem on $\tilde{h}$ to conclude that, $\tilde{h}$ attains its extrema in $\mathbb{T}^0$, i.e. there exists an \(M>0\) such that 
\begin{equation}
    \label{extreme_value}
    \left|\tilde{h}(t,\ x,\ y)\right|\leq M\text{ for }t\in[1,\ T+1],\ (x,\ y)\in\mathbb{R}\text{.}
\end{equation}

Finally, we define the following auxiliary problem corresponding to \eqref{Problem} using the function $\tilde{h}$ as described above:
\begin{equation}
    \label{regular_problem}
    \begin{cases}
        u^{\Delta\Delta}(t-1)+\tilde{h}(t,\ u(t),\ u^{\Delta}(t-1))=0\\
        u^{\Delta}(0)=0\\
        u(T+2)=g(T+2)
    \end{cases}
\end{equation}

\begin{rem}
The solution $u(t);~\alpha(t) \le u(t) \le \beta(t)$ to the auxiliary problem \eqref{regular_problem} is equivalent to the solution $u(t)$ of the given problem \eqref{Problem}, since $\tilde{h}\equiv h$ when $\alpha (t) \le t \le \beta (t)$.
\end{rem}

\subsection*{Step 2: Construction of solution operator}

In this step, we define the corresponding solution operator for the auxiliary problem \eqref{regular_problem} such that a fixed point of the solution operator turns out to be a solution of \eqref{regular_problem}. To define the solution operator, first we define the underlying function space and the appropriate norm as follows:

\begin{defi}
 Define the space \(E := \{u\colon[0,\ T+2]\to\mathbb{R},\ u^{\Delta}(0)=0,\ u(T+2)=cg(T+2)\}\) with the associated distance function \[\|u\|:=\max\{u(t)\colon t\in[1,\ T+1]\}\] where $c$ is an arbitrary positive constant.  
\end{defi}

\begin{lemma}
$\langle E, \|u\|\rangle$ is a Banach space.
\end{lemma}

\begin{proof}
{\bf Criterion 1: Show that $E$ is a linear space.}
This follows from the fact that $u \in E$ is a real-valued function and the forward difference operator $\Delta$ is linear.

{\bf Criterion 2: Show that $\|\cdot\|$ defined above is a norm.} Note that since $g$ is a positive function and $c$ is an arbitrary constant, $u(T+2) >0$.  Therefore, $\|u\| \ge 0$, for any $u \in E $ and $\|u\|=0$ if and only if $u \equiv 0$. Triangle inequality follows from the fact that maximum of sum of two functions is less than the sum of the maximums of the functions.

{\bf Criterion 3: To show $\langle E, \|u\|\rangle$ is complete.}
Observe that, $u \in E$ is a real-valued function and the completeness of $\mathbb{R}$ implies the completeness of $E$. 
Hence we can conclude that $\langle E, \|u\|\rangle$ is a Banach space.
\end{proof}

Now, we define the solution operator \(\mathcal{T}\colon E\to E\) by
\begin{equation}
    \label{t_operator}
    (\mathcal{T}u)(t)=cg(T+2)+\sum_{s=t}^{T+1}\left(\sum_{i=1}^s
    \tilde{h}(i,\ u(i),\ u^{\Delta}(i-1))\right),\ t\in[0,\ T+2]
\end{equation}

\subsection*{Step 3: Application of Brouwer's Fixed Point Theorem}

Continuity of $g$ and $\tilde{h}$ implies the continuity of the operator \(\mathcal{T}\). Moreover, continuity of $g$ over a closed set, \eqref{extreme_value} and \eqref{t_operator} imply that if
\[r\geq\sum_{s=1}^{T+1}sM\]
Then \(\mathcal{T}(\overline{B(r)})\subseteq\overline{B(r)}\), where
\(B(r)=\{u\in E\colon\|u\|<r\}\).
Therefore, by the Brouwer Fixed Point Theorem, \(\exists u\in\overline{B(r)}\) such that \(u=\mathcal{T}u\).

\subsection*{Step 4: \(u\) is a fixed point of \(\mathcal{T}\) iff \(u\) is a solution of \eqref{regular_problem}}

\subsubsection*{Forward Direction \((\Rightarrow)\)} 

Let us assume \(u=\mathcal{T}u\). Without loss of generality, we can assume $c=1$. Since $\mathcal{T}: E \to E$, hence \(u\in E\) and thus, satisfies the boundary conditions in \eqref{regular_problem}. 
Moreover,  for  $t\in[1,\ T+1]$, we have,
\begin{align*}
    u^{\Delta}(t-1) &= u(t)-u(t-1)\\
                  &= g(T+2)+\sum_{s=t}^{T+1}\sum_{i=1}^s\tilde{h}(i,\ u(i),\ u^{\Delta}(i-1))\\
                  &\quad \quad -\left(g(T+2)+\sum_{s=t-1}^{T+1}\sum_{i=1}^s\tilde{h}(i,\ u(i),\ u^{\Delta}(i-1))\right)\\
                  &= -\sum_{i=1}^{t-1}\tilde{h}(i,\ u(i),\ u^{\Delta}(i-1))\\
   \Rightarrow u^{\Delta \Delta}(t-1) &= u^{\Delta}(t)-u^{\Delta}(t-1)\\
                        &=-\tilde{h}(t,\ u(t),\ u^{\Delta}(t-1)).
\end{align*}
Therefore, we can conclude that $u$ is a solution of \eqref{regular_problem} if $u$ is a fixed point of the operator $\mathcal{T}$.

\subsubsection*{Backward Direction \((\Leftarrow)\)}

Next, let us assume \(u\) solves \eqref{regular_problem} with $c=1$. Then \(u\in E\) and \[ u^{\Delta \Delta}(t)=-\tilde{h}(t,\ u(t),\ u^{\Delta}(t-1)).\] We will show $u(t)= (\mathcal{T}u)(t)$, for $t \in [1, T+1]$. First, continuing via mathematical induction,  we can see, \begin{equation}
    \label{delta_u}
    u^{\Delta}(t)=-\sum_{i=1}^t\tilde{h}(i,\ u(i),\ u^{\Delta}(i-1))
\end{equation}

Indeed, for \(t=1\),
\begin{align*}
    u^{\Delta\Delta}(0) &= u^{\Delta}(1)-u^{\Delta}(0)\\
    -\tilde{h}(1,\ u(1),\ u^{\Delta}(0)) &= u^{\Delta}(1)-0\\
    u^{\Delta}(1) &= -\tilde{h}(1,\ u(1),\ u^{\Delta}(0));
\end{align*}
for $t =2$,
\begin{align*}
    u^{\Delta\Delta}(1) &= u^{\Delta}(2)-u^{\Delta}(1)\\
    \Rightarrow-\tilde{h}(1,\ u(1),\ u^{\Delta}(0)) &= u^{\Delta}(2)+\tilde{h}(1,\ u(1),\ u^{\Delta}(0))\\
    \Rightarrow u^{\Delta}(2) &= -(\tilde{h}(1,\ u(1),\ u^{\Delta}(0))+\tilde{h}(2,\ u(2),\ u^{\Delta}(1))); 
\end{align*}
for $t=3$,
\begin{align*}
    u^{\Delta\Delta}(2) &= u^{\Delta}(3)-u^{\Delta}(2)\\
    \Rightarrow -\tilde{h}(2,\ u(2),\ u^{\Delta}(1)) &= u^{\Delta}(3)
    +\tilde{h}(1,\ u(1),\ u^{\Delta}(0))+\tilde{h}(2,\ u(2),\ u^{\Delta}(1))\\
    \Rightarrow u^{\Delta}(3) &= -\left(\tilde{h}(1,\ u(1),\ u^{\Delta}(0))
    +\tilde{h}(2,\ u(2),\ u^{\Delta}(1))+\tilde{h}(3,\ u(3),\ u^{\Delta}(2))\right);
\end{align*} 
and so on.

Next, we again apply mathematical induction, starting backwards, to show, for $t=T+1$,
\begin{align*}
    u^{\Delta}(T+1) &= u(T+2)-u(T+1)\\
    \Rightarrow -\sum_{i=1}^{T+1}\tilde{h}(i,\ u(i),\ u^{\Delta}(i-1)) &= g(T+2)-u(T+1)\\
    \Rightarrow u(T+1) &= g(T+2)+\sum_{i=1}^{T+1}\tilde{h}(i,\ u(i),\ u^{\Delta}(i-1));
\end{align*}
for $t=T$,
\begin{align*}
    u^{\Delta}(T) &= u(T+1)-u(T)\\
    \Rightarrow -\sum_{i=1}^{T}\tilde{h}(i,\ u(i),\ u^{\Delta}(i-1)) &= g(T+2)
    +\left(\sum_{i=1}^{T+1}\tilde{h}(i,\ u(i),\ u^{\Delta}(i-1))\right)-u(T)\\
    \Rightarrow u(T) &= g(T+2)+\left(\sum_{i=1}^{T+1}\tilde{h}(i,\ u(i),\ u^{\Delta}(i-1))\right)\\
    & \quad \quad +\left(\sum_{i=1}^{T}\tilde{h}(i,\ u(i),\ u^{\Delta}(i-1))\right)\\
         &= g(T+2)+\sum_{s=T}^{T+1}\sum_{i=1}^s\tilde{h}(i,\ u(i),\ u^{\Delta}(i-1));
\end{align*}

and so on. Thus continuing on via induction and using the vacuous cases that \(\displaystyle\sum_i^0=0\) and
\(\displaystyle\sum_{T+2}^{T+1}=0\) we get,
\begin{equation}
    \label{u(t)}
    u(t)=g(T+2)+\sum_{s=t}^{T+1}\sum_{i=1}^{s}\tilde{h}(i,\ u(i),\ u^{\Delta}(i-1))
    \text{,\quad}t\in[0,\ T+2]
\end{equation}

\subsection*{Step 5:}

Finally, we show that, \(u\) satisfies
\[\alpha(t)\leq u(t)\leq\beta(t)\text{,\quad}t\in[0,\ T+2]\]
We will prove case \(u(t)\leq\beta(t)\) and the proof for the case $\alpha(t) \le u(t)$ can be done similarly.
Let us consider the function \(v(t)=u(t)-\beta(t)\). Let us assume that \(\max\{v(t)\colon t\in[0,\ T+2]\}=v(\ell)>0\). Boundary conditions described in \eqref{regular_problem} and \eqref{beta} imply that $\ell \notin \{0,T+2\}$. Therefore, \(\ell\in[1,\ T+1]\). Consequently, \(v(\ell+1)\leq v(\ell)\) and \(v(\ell)\geq v(\ell-1)\) which implies \(    v^{\Delta}(\ell)\leq 0\) and \(v^{\Delta}(\ell-1)\geq 0\).
Hence,
\begin{equation}
    \label{contradiction}
    v^{\Delta \Delta}(\ell-1)\leq 0 \Rightarrow u^{\Delta \Delta}(\ell-1)\leq\beta^{\Delta\Delta}(\ell-1).
\end{equation}

On the other hand, utilizing the definition of $\tilde{h}(t, x, x-z)$ for $t=\ell,~x=u(\ell),~z=u(\ell -1)$ and the fact \(h\) is non-increasing in its third variable (see {\bf (A2)}),
we obtain, 
\begin{align*}
    u^{\Delta\Delta}(\ell-1)-\beta^{\Delta\Delta}(\ell-1) &=
    -\tilde{h}(\ell,\ u(\ell),\ u^{\Delta}(\ell-1))-\beta^{\Delta\Delta}(\ell-1)\\
                                                     &= -h(\ell,\ \beta(\ell),\ \beta(\ell)-\sigma(\ell,\ u(\ell-1)))+\frac{v(\ell)}{v(\ell)+1}-\beta^{\Delta\Delta}(\ell-1)\\
                                                     &\geq -h(\ell,\ \beta(\ell),\ \beta^{\Delta}(\ell-1)+\frac{v(\ell)}{v(\ell)+1}-\beta^{\Delta\Delta}(\ell-1)\\
                                                     &\geq\frac{v(\ell)}{v(\ell)+1}\\
                                                     &> 0
\end{align*}
which contradicts \eqref{contradiction}. Therefore, \(u(t)\leq\beta(t)\) for \(t\in[0,\ T+2]\) and this completes the proof of Theorem \ref{Theorem}.

\medskip

\noindent \textbf{Shalmali Bandyopadhyay}\\  
The University of Tennessee at Martin\\
Martin, TN\\
E-mail: \texttt{sbandyo5@utm.edu}

\medskip

\noindent \textbf{Kyle Byassee}\\  
The University of Tennessee at Martin\\
Martin, TN\\
E-mail: \texttt{kylhbyas@ut.utm.edu}

\medskip

\noindent \textbf{Curt Lynch}\\  
The University of Tennessee at Martin\\
Martin, TN\\
E-mail: \texttt{curjlync@ut.utm.edu}


\begin{thebibliography}{6}

\bibitem{AgOR1} 
R.P. Agarwal and D. O'Regan, 
\emph{Singular discrete boundary value problems}, 
Appl. Math. Lett., 12 (1999), 127--131.

\bibitem{BoPe} 
M. Bohner and A. Peterson, 
\emph{Dynamic equations on time scales}, 
Birkhäuser, 2001.

\bibitem{Ku1} 
C. Kunkel, 
\emph{Singular second order boundary value problems for differential equations}, 
Proceedings of Neural, Parallel, and Scientific Computations, 3 (2006), 119--124.

\bibitem{KuLa} 
C. Kunkel and A. Lancaster, 
\emph{Positive solutions to singular second-order boundary value problems for dynamic equations}, 
Involve, 12(6) (2019), 1069--1080.

\bibitem{KuThesis} 
C.J. Kunkel, 
\emph{Positive solutions of singular boundary value problems}, 
ProQuest LLC, Ann Arbor, MI, 2007. Thesis (Ph.D.)--Baylor University.

\bibitem{Ze} 
E. Zeidler, 
\emph{Nonlinear functional analysis and its applications I: fixed-point theorems}, 
Springer-Verlag, 1986.

\end{thebibliography}
\end{document}